\documentclass[a4paper,12pt]{amsart}

\usepackage{graphicx}
\usepackage[english]{babel}

\usepackage[T1]{fontenc}
\usepackage[utf8]{inputenc}
\usepackage{csquotes}
\usepackage{verbatim}
\usepackage{todonotes}
\usepackage{amsmath}
\usepackage{amsthm}
\usepackage{amsfonts, amssymb, fancybox, multicol, array, tabularx, amscd}
\usepackage{enumerate}
\usepackage{upgreek}
\usepackage[shortlabels]{enumitem}
\usepackage{psfrag}

\newtheorem{theorem}{Theorem}[section]
\newtheorem{lemma}[theorem]{Lemma}
\newtheorem{proposition}[theorem]{Proposition}

\newtheorem{remark}[theorem]{Remark}

\numberwithin{equation}{section}

\theoremstyle{definition}
\newtheorem{definition}[theorem]{Definition}

\theoremstyle{remark}

\newcommand{\R}{\mathbb{R}}
\newcommand{\N}{\mathbb{N}}

\newcommand{\C}{\mathbb{C}}

\newcommand{\rot}{R_{\frac{\pi}{2}}}
\newcommand{\matrici}{\mathbb{R}^{2 \times 2}}

\newcommand{\rotazioni}{SO(2)}

\def\dsp{\displaystyle}

\newcommand{\va}[1]{\left| #1 \right|}  
\newcommand{\nor}[1]{\left\| #1 \right\|} 
\newcommand{\holder}[1]{{\left[ #1 \right]}_{C^{\alpha} (\conj{\Omega})     }}

\newcommand{\conj}[1]{\overline{#1}}  
\DeclareMathOperator{\dist}{dist}
\DeclareMathOperator{\diag}{diag}
\DeclareMathOperator{\rank}{rank}
\DeclareMathOperator{\Div}{div}
\DeclareMathOperator{\spt}{spt}
\DeclareMathOperator{\tr}{Tr}

\newcommand{\Om}{\Omega}

\newcommand{\target}{T}
\newcommand{\uno}{T_1}  
\newcommand{\due}{T_2}  
\newcommand{\cK}{c_K}    

\newcommand{\step}{\mathcal{S}}

\newcommand{\Rmn}{\R^{m \times n}}
\newcommand{\measures}{\mathcal{M}(\matrici)}
\newcommand{\laminates}{\mathcal{L}(\matrici)}
\newcommand{\lam}{\mathcal{L}(\matrici)}
\newcommand{\Msdet}{ SL_{\rm sym}(2)}

 \title[Higher gradient integrability for two-phase elliptic equations]{Optimal lower exponent for the higher gradient integrability of solutions to two-phase elliptic equations in two dimensions}
 
 \author[S. Fanzon]
 {Silvio Fanzon}
 \address[Silvio Fanzon]{University of Sussex, Department of Mathematics, Pevensey 2 Building, Falmer Campus,
Brighton BN1 9QH, United Kingdom}
 \email{S.Fanzon@sussex.ac.uk}

\author[M. Palombaro]
 {Mariapia Palombaro}
 \address[Mariapia Palombaro]{University of Sussex, Department of Mathematics, Pevensey 2 Building, Falmer Campus,
Brighton BN1 9QH, United Kingdom}
 \email{M.Palombaro@sussex.ac.uk}

\begin{document}

\begin{abstract}
We study the higher gradient integrability of distributional solutions $u$ to the equation 
$\Div(\sigma \nabla u) = 0$ in dimension two, in the case when the essential range of $\sigma$ consists of only two elliptic matrices, 
i.e., $\sigma\in\{\sigma_1, \sigma_2\}$ a.e. in $\Om$.  
In \cite{npp}, for every pair of elliptic matrices $\sigma_1$ and $\sigma_2$, exponents $p_{\sigma_1,\sigma_2}\in(2,+\infty)$ and $q_{\sigma_1,\sigma_2}\in (1,2)$ 
have been characterised so that if $u\in W^{1,q_{\sigma_1,\sigma_2}}(\Om)$ is solution to the elliptic equation then  
$\nabla u\in L^{p_{\sigma_1,\sigma_2}}_{\rm weak}(\Om)$ and the optimality of the upper exponent $p_{\sigma_1,\sigma_2}$ has been proved.  
In this paper we complement the above result by proving the optimality of the lower exponent $q_{\sigma_1,\sigma_2}$. 
Precisely, we show that for every arbitrarily small $\delta$, one can find a particular microgeometry, i.e., an arrangement of the sets $\sigma^{-1}(\sigma_1)$ and $\sigma^{-1}(\sigma_2)$, 
for which there exists a solution $u$ to the corresponding elliptic equation such that $\nabla u \in L^{q_{\sigma_1,\sigma_2}-\delta}$, but $\nabla u \notin L^{q_{\sigma_1,\sigma_2}}$. The existence of such optimal microgeometries is achieved by convex integration methods, adapting the geometric constructions provided in \cite{afs} in the isotropic case to the present setting.


\end{abstract}

\maketitle
\tableofcontents

\section{Introduction}\label{introduction}
\noindent
Let $\Omega \subset \R^2$ be a bounded open domain and let $\sigma  \in L^{\infty} (\Omega; \matrici)$ be uniformly elliptic, i.e.,
$$
\sigma \xi \cdot \xi \geq \lambda |\xi|^{2} \text{ for every } \xi \in \R^2 \text{ and for a.e. } x \in \Omega,
$$ 
for some $\lambda >0$. 
We study the gradient integrability of distributional solutions $u\in W^{1,1}(\Om)$ to
\begin{equation} \label{the pde}
\Div (\sigma(x) \nabla u (x)) = 0   \quad \text{ in }\,\, \Omega,	
\end{equation}
in the case when the essential range of $\sigma$ consists of only two matrices, say $\sigma_1$ and $\sigma_2$.
It is well-known from Astala's work \cite{a} that there exist exponents $q$ and $p$, with $1<q<2<p$, such that 
if $u\in W^{1,q}(\Om ; \R)$ is solution to \eqref{the pde}, then $\nabla u\in L^{p}_{\rm weak}(\Om ;\R)$. 
In \cite{npp} the optimal exponents $p$ and $q$ have been characterised for every pair of elliptic matrices $\sigma_1$ and $\sigma_2$.
Denoting by $p_{\sigma_1,\sigma_2}$ and $q_{\sigma_1,\sigma_2}$ such exponents, whose precise formulas are recalled in Section \ref{formule},
 we summarise the result of \cite{npp} in 
the following theorem.
\begin{theorem}\cite[Theorem 1.4 and Proposition 4.2]{npp}\label{upper-optimality}
Let $\sigma_1, \sigma_2 \in \matrici$ be elliptic.
\begin{itemize}
\item[i)] If $\sigma\in L^{\infty}(\Om;\{\sigma_1,\sigma_2\})$ and $u\in  W^{1,q_{\sigma_1,\sigma_2}}(\Om)$ solves \eqref{the pde}, 
then $\nabla u\in L^{p_{\sigma_1,\sigma_2}}_{\rm weak}(\Om;\R^2)$.\\[-2mm]
\item[ii)] There exists $\bar\sigma\in L^{\infty}(\Om;\{\sigma_1,\sigma_2\})$ and a weak solution 
$\bar{u}\in W^{1,2}(\Om)$ to \eqref{the pde} with $\sigma=\bar\sigma$,  satisfying 
affine boundary conditions and such that $\nabla \bar{u}\notin L^{p_{\sigma_1,\sigma_2}}(\Om ;\R^2)$.
\end{itemize}
\end{theorem}
Theorem \ref{upper-optimality} proves the optimality of the upper exponent $p_{\sigma_1,\sigma_2}$.
The objective of this paper is to complement this result by proving  the optimality of the lower exponent $q_{\sigma_1,\sigma_2}$.
As shown in \cite{npp} (and recalled in Section \ref{formule}), there is no loss of generality in assuming that 
\begin{equation}\label{speciali}
\sigma_1 = \diag (1/K,1/S_1),\quad \sigma_2 = \diag(K,S_2), 
\end{equation}
with 
\begin{equation} \label{ellipt real}
K>1  \qquad  \text{and}  \qquad \frac{1}{K} \leq S_j \leq K  \, , \quad j=1,2 \,.  
\end{equation}
Thus it suffices to show optimality for this class of coefficients, for which the exponents 
$p_{\sigma_1,\sigma_2}$ and $q_{\sigma_1,\sigma_2}$ read as
\begin{equation}\label{formula-pq}
 q_{\sigma_1,\sigma_2}= \frac{2K}{K+1}, \quad p_{\sigma_1,\sigma_2}= \frac{2K}{K-1}.
\end{equation}
Our main result is the following
%

\begin{theorem} \label{main theorem}
Let $\sigma_1,\sigma_2 $ be defined by \eqref{speciali} for some $K>1$ and $S_1, S_2 \in [1/K , K]$. 
 There exist 
coefficients $\sigma_n \in L^{\infty}(\Omega,\{ \sigma_1; \sigma_2 \})$, 
exponents $p_n \in \left[1,\frac{2K}{K+1} \right]$, functions $u_n \in W^{1,1} (\Omega;\R)$
such that 

\begin{align} 
\label{pde3}
&\begin{cases}
\Div (\sigma_n (x) \nabla u_n (x)) = 0  &  \text{ in } \quad \Omega \,,\\
u_n (x) = x_1							& \text{ on }  \quad \partial \Omega \,,
\end{cases}\\
&\nabla u_n \in L^{p_n}_{\rm weak}(\Omega;\R^2), \quad p_n \to \frac{2K}{K+1}, \\
&\nabla u_n \notin L^{\frac{2K}{K+1}}(\Omega;\R^2).
\end{align}
In particular $u_n \in W^{1,q} (\Omega;\R)$ for every $q < p_n$, but $\int_{\Omega} {|\nabla u_n|}^{\frac{2K}{K+1}} \, dx= \infty$.
\end{theorem}
Theorem \ref{main theorem} was proved in \cite{afs} in the case of isotropic coefficients, namely for $\sigma_1=\frac{1}{K} I$ and 
$\sigma_2= KI$. We follow the method developed in \cite{afs}, which relies on convex integration and provides an explicit construction 
of the sequence $u_n$. The adaptation of such method to the present context is definitely non-trivial due to the anisotropy of the coefficients. 
%

\section{Connection with the Beltrami equation and explicit formulas for the optimal exponents}
\label{formule}
For the reader's convenience we recall in this section how to reduce to the case \eqref{speciali} starting from any pair
$\sigma_1,\sigma_2 $. We will also give 
the explicit formulas for  $p_{\sigma_1,\sigma_2}$ and $q_{\sigma_1,\sigma_2}$.

It is well-known that a solution $u\in W^{1,q}_{loc}$, $q\geq 1$, to the elliptic equation \eqref{the pde} can be regarded as
the real part of a complex map  
$\dsp f:\Om\mapsto\C$ which is a $W^{1,q}_{loc}$  solution to a {\it Beltrami equation}. 
Precisely,  if $v$ is such that 
\begin{equation}\label{stream}
\rot^T \nabla  v =  \sigma \nabla u, \quad \rot := \left(\begin{array}{cc}
0&-1\\
1&0
\end{array}\right),
\end{equation} 
then $f:=u+iv$ solves the equation
\begin{equation}\label{beltrami}
 f_{\bar{z}}=\mu \, f_{z}+ \nu \, \overline{ f_{z}}\quad \text{a.e.  in }\Om\,,
 \end{equation}
 where the so called complex dilatations $\mu$ and $\nu$, both belonging to $L^{\infty}(\Om;\mathbb C)$, are given by 
 \begin{equation}\label{mu-nu(sigma)}
 \mu=\frac{\sigma_{22}-\sigma_{11}-i(\sigma_{12}+\sigma_{21})}{1+\tr \, \sigma+\det\sigma}\,,
 \quad
 \nu=\frac{1-\det\sigma+i(\sigma_{12}-\sigma_{21})}{1+\tr \, \sigma+\det\sigma}\,,
\end{equation}
\noindent 
and satisfy the ellipticity condition 
\begin{equation}\label{ellipticity-munu}
\| |\mu|+|\nu| \|_{L^\infty}< 1 \,.
\end{equation}
The ellipticity \eqref{ellipticity-munu}  is often expressed in a different form. Indeed, it implies that there exists 
$0\leq k<1$ such that $\| |\mu|+|\nu| \|_{L^\infty}\leq k< 1$ or equivalently that
\begin{equation}\label{ellipticity-munu-K}
\| |\mu|+|\nu| \|_{L^\infty}\leq \frac{K-1}{K+1} \,,
\end{equation}
for some $K>1$.
Let us recall that weak solutions to \eqref{beltrami}, \eqref{ellipticity-munu-K} are called $K$-quasiregular mappings. 
Furthermore, we can express $\sigma$ as a function of $\mu, \, \nu$ 
inverting the  algebraic system \eqref{mu-nu(sigma)},  
\begin{equation}\label{sigma(mu-nu)}
\sigma =
\left(
\begin{array}{ll}
\frac{|1-\mu|^2-|\nu|^2}{|1+\nu|^2-|\mu|^2} &  \frac{2\Im(\nu-\mu)}{|1+\nu|^2-|\mu|^2}\\
\, & \, \\
\frac{-2\Im(\nu+\mu)}{|1+\nu|^2-|\mu|^2} & \frac{|1+\mu|^2-|\nu|^2}{|1+\nu|^2-|\mu|^2}
\end{array}
\right)\,.
\end{equation}
Conversely, if $f$ solves \eqref{beltrami} with $\mu,\nu \in L^{\infty}(\Om,\mathbb C)$ satisfying 
\eqref{ellipticity-munu},
then its real part is solution to the elliptic equation \eqref{the pde} with $\sigma$ defined by \eqref{sigma(mu-nu)}.
Notice that $\nabla f$ and $\nabla u$ enjoy the same integrability properties.
Assume now that $\sigma:\Om\to\{\sigma_1,\sigma_2\}$ is a two-phase elliptic coefficient and $f$ is solution to \eqref{beltrami}-\eqref{mu-nu(sigma)}. 
Abusing notation, we identify $\Om$ with a subset of $\R^2$ and $f=u+iv$ with the real mapping $f=(u,v):\Om\to \R^2$. 
Then, as shown in \cite{npp},  
one can find matrices $A,B\in \Msdet$ (with $ \Msdet$ denoting the set of invertible matrices with determinant equal to one) depending only on 
$\sigma_1$ and $\sigma_2$, such that, setting 
\begin{equation}\label{trasformatedet}
\tilde f(x):= A^{-1} f(Bx), 
\end{equation}
one has that the function $\tilde f$   solves the new Beltrami equation 
\begin{equation*}
 \tilde f_{\bar{z}}=\tilde\mu \, f_{z}+ \tilde\nu \, \overline{ \tilde f_{z}} \quad \hbox{a.e.  in } B(\Om),
\end{equation*}
and the corresponding $\tilde\sigma: B(\Om)\to \{\tilde\sigma_1,\tilde\sigma_2\}$ defined by \eqref{sigma(mu-nu)} is of 
the form \eqref{speciali}:
\begin{equation*}
\tilde\sigma_1 = \diag (1/K,1/S_1),\quad \tilde\sigma_2 = \diag(K,S_2), \quad K>1, \quad S_1,S_2 \in [1/K ,K]\,.
\end{equation*}
The results in \cite{a} and \cite{pv} imply that if $\tilde f\in W^{1,q}$, with $q\geq \frac{2K}{K+1}$, then 
 $\nabla\tilde f\in L^{ \frac{2K}{K-1}}_{\rm weak}$; in particular,
$\tilde f\in W^{1,p}$ for each $p< \frac{2K}{K-1}$.
Clearly $\nabla\tilde f$ enjoys the same integrability properties as $\nabla f$ and $\nabla u$.

Finally, we recall the formula for $K$ which will yield the optimal exponents. 
Denote by 
$d_1$ and $d_2$  the determinant of the symmetric part of $\sigma_1$ and 
$\sigma_2 $ respectively, 
$$
d_i:=\det\Big(\frac{\sigma_i + \sigma_i^T}{2}\Big) \,, \quad i=1,2 \,,
$$
and by $(\sigma_i)_{jk}$ the $jk$-entry of $\sigma_i$.
Set
\begin{align*}
 m: &= \frac{1}{\sqrt{d_1 d_2}}\left[(\sigma_2)_{11}  (\sigma_1)_{22}  +
(\sigma_1)_{11}  (\sigma_2)_{22}   -
\frac{1}{2} 
\Big((\sigma_2)_{12}+(\sigma_2)_{21}\Big)
\Big((\sigma_1)_{12}+(\sigma_1)_{21}\Big) \right] \,, \\
 n: &=  \frac{1}{\sqrt{d_1 d_2}}
\left[
\det\sigma_1 + \det\sigma_2 -\frac{1}{2} 
\Big((\sigma_1)_{21} - (\sigma_1)_{12}\Big)
\Big((\sigma_2)_{21} - (\sigma_2)_{12}\Big)
\right] \,.
\end{align*}
Then 
\begin{equation}\label{formula-K}
K=
\left(\frac{m + \sqrt{m^2 - 4}}{2} \right)^{\frac{1}{2}}\left(\
\frac{n + \sqrt{n^2 -4}}{2}\right)^{\frac{1}{2}}.
\end{equation}
Thus,  for any pair 
of elliptic matrices $\sigma_1,\sigma_2 \in\matrici$, the explicit formula for the optimal exponents $p_{\sigma_1,\sigma_2}$ and $q_{\sigma_1,\sigma_2}$ 
are obtained 
by plugging \eqref{formula-K} into \eqref{formula-pq}.

\section{Preliminaries}

\subsection{Conformal coordinates}

For every real matrix $A \in \matrici$, 
\[
A= \left(\begin{matrix}
 	a_{11}  &  a_{12} \\
 	a_{21}  &  a_{22} \\
 \end{matrix} \right),
\]
we write $A=(a_+, a_-)$, where $a_+, a_- \in \C$ denote its conformal coordinates. By identifying any vector
$v=(x,y) \in \R^2$ with the complex number $v=x + i y $, conformal coordinates 
are defined by the identity
\begin{equation} \label{conf coord}
Av = a_+ v + a_- \conj{v} \,.
\end{equation}
Here $\conj{v}$ denotes the complex conjugation. 
From \eqref{conf coord} we have relations
\begin{equation} \label{real to conf}
a_+ = \frac{a_{11}+a_{22}}{2} + i \, \frac{a_{21}-a_{12}}{2}\,, 	\qquad \qquad
a_-  = \frac{a_{11}-a_{22}}{2} + i \, \frac{a_{21}+a_{12}}{2}\,,
\end{equation}
and, conversely,
\begin{equation} \label{conf to real}
\begin{aligned}
a_{11} & = \Re a_+ + \Re a_- \,,      \qquad \qquad   a_{12} = - \Im a_+ + \Im a_- \,,	 \\
a_{21} & = \Im a_+ + \Im a_- \,,       \qquad \qquad   a_{22} = \Re a_+ - \Re a_- 	\,.
\end{aligned}	
\end{equation}
Here $\Re z$ and $ \Im z$ denote the real and imaginary part of $z \in \C$ respectively.
We recall that 
\begin{equation} \label{multiplication}
AB= (a_+ b_+ + a_- \conj{b}_-, a_+ b_- + a_- \conj{b}_+)\,,
\end{equation}
and $\tr A = 2 \Re a_+$. Moreover 
\begin{equation} \label{formulae}
\begin{gathered}
\det (A)  = \va{a_+}^2 - \va{a_-}^2\,,	 \\
\va{A}^2  = 2 \va{a_+}^2 + 2 \va{a_-}^2 \,, \\
\nor{A}  = \va{a_+} + \va{a_-}\,,
\end{gathered}
\end{equation}
where $\va{A}$ and $\nor{A}$ denote the Hilbert-Schmidt and the operator norm, respectively. 

We also define the second complex dilatation of the map $A$ as
\begin{equation} \label{second dilatation}
\mu_A := \frac{a_-}{\conj{a}_+} 	\,,
\end{equation}
and the distortion 
\begin{equation} \label{distortion}
K(A) := \va{ \frac{1 + \va{\mu_A}}{1-\va{\mu_A}}   }= \frac{\nor{A}^2}{\va{\det (A)}} 	 \,.
\end{equation}
The last two quantities measure how far $A$ is from being conformal. 
Following the notation introduced in \cite{afs}, we define
\begin{equation} \label{conformal set}
E_{\Delta} := \{A= (a,\mu \, \conj{a}) \, \colon \,  a \in \C , \, \mu \in \Delta \}
\end{equation}
for a set $\Delta \subset \C \cup \{ \infty \}$; namely, $E_{\Delta}$ is the set of matrices with the second complex dilatation belonging to $\Delta$.
In particular $E_0$ and $E_{\infty}$ denote the set of conformal and anti-conformal matrices respectively. From \eqref{multiplication} we have that
$E_{\Delta}$ is invariant under precomposition by conformal matrices, that is
\begin{equation} \label{conformal invariance}
E_{\Delta} =E_{\Delta} A  \qquad \text{for every} \qquad A \in E_0 \smallsetminus \{0\} \,.
\end{equation}

\subsection{Convex integration tools}

We denote by $\measures$ the set of signed Radon measures on $\matrici$ having finite
mass. By the Riesz's representation theorem we can identify $\measures$
with the dual of the space $C_0 (\Rmn)$. Given $\nu \in \measures$ we define its 
\textit{barycenter} as
\[
\conj{\nu} := \int_{\matrici} A \, d\nu(A) \,.
\]  
We say that a map $f \in C(\conj{\Omega}; \R^2)$ is \textit{piecewise affine} if there exists a countable family of pairwise disjoint open subsets $\Omega_i \subset \Omega$ with $\va{\partial \Omega_i}=0$ and 
\[
\va{ \Omega \smallsetminus \bigcup_{i=1}^{\infty} \Omega_i  } =0  \,,
\]
such that $f$ is affine on each $\Omega_i$. Two matrices $A , B \in \matrici$
such that $\rank (B-A)=1$ are said to be \textit{rank-one connected} and the measure
$\lambda \delta_A + (1- \lambda) \delta_B \in \measures$ with $\lambda \in [0,1]$ is called a \textit{laminate of first order}.

\begin{definition}
The family of \textit{laminates of first order} $\laminates$ is the smallest family of 
probability measures in $\measures$	satisfying the following conditions: 
\begin{enumerate}[\indent (i)]
\item 	$\delta_A \in \laminates$ for every $A \in \matrici \,$;
\item assume that $\sum_{i=1}^N \lambda_i \delta_{A_i} \in \laminates$ and $A_1=\lambda B + (1-\lambda)C$ with $\lambda \in [0,1]$ and $\rank(B-C)=1$. Then the probability measure
\[
\lambda_1 (\lambda \delta_B + (1-\lambda ) \delta_C) + \sum_{i=2}^N \lambda_i \delta_{A_i} 
\]
is also contained in $\laminates$.
\end{enumerate}
\end{definition}
The process of obtaining new measures via (ii) is called \textit{splitting}.
The following proposition provides a fundamental tool to solve differential inclusions by means of convex integration (see e.g. \cite[Proposition 2.3]{afs} for a proof).

\begin{proposition} \label{gradienti}
Let $\nu = \sum_{i=1}^N \alpha_i \delta_{A_i} \in \laminates$ be a laminate of finite order with barycenter $\conj{\nu}=A$, that is $A= \sum_{i=1}^N \alpha_i A_i$ with $\sum_{i=1}^N \alpha_i=1$. Let $\Omega \subset \R^2$ be a bounded open set, $\alpha\in (0,1)$ and $0<\delta < \min \va{A_i-A_j}/2$. Then there exists a piecewise affine Lipschitz map $f \colon \Omega \to \R^2$ such that
\begin{enumerate}[\indent (i)]
\item $f(x)=Ax$ on $\partial \Omega$, \smallskip
\item $\holder{f-A} < \delta$ , \smallskip
\item $\va{ \{ x \in \Omega \, \colon \, \va{\nabla f (x) - A_i} < \delta \}  } = \alpha_i \va{\Omega}$, \smallskip
\item $\dist (\nabla f (x), \spt \nu) < \delta\,$ a.e. in $\Omega$. 	\smallskip
\end{enumerate} 
\end{proposition}

\subsection{Weak $L^p$ spaces}

We recall the definition of weak $L^{p}$ spaces. Let 
$f \colon \Omega \to \R^2$ be a Lebesgue measurable function. 
Define the distribution function of $f$ as 
\[
\lambda_f \colon (0,\infty) \to [0,\infty] \quad \text{with} \quad \lambda_f (t) := \va{ \{  x \in \Omega \, \colon \, |f(x)|>t    \}    } \,.
\]
Let $1\leq p<\infty$, then the following formula holds
\begin{equation} \label{cavalieri}
\int_{\Omega} \va{f(x)}^p \, dx = p \int_0^\infty t^{p-1} \lambda_f (t) \, dt  \,.
\end{equation}
Define the quantity
\[
[f]_p :={ \left(   \sup_{t>0} \, t^p \lambda_f (t)  \right)  }^{1/p}   
\]
and the weak $L^p$ space as
\[
L^p_{\rm weak} (\Omega; \R^2) := \left\{ f \colon \Omega \to \R^2 \, \colon \, f \, \text{ measurable}, \, [f]_p <\infty \right\} \,.
\]
$L^p_{\rm weak}$ is a topological vector space and by Chebyshev's inequality we have
$[f]_p \leq \nor{f}_{L^p}$. In particular this implies 
$L^p \subset L^p_{\rm weak}$.

\section{Proof of Theorem \ref{main theorem}}

For the rest of this paper, $\sigma_1$ and $\sigma_2$ are as in \eqref{speciali}-\eqref{ellipt real}. 
We start by rewriting \eqref{the pde} as a differential inclusion. To this end, define the sets
\begin{equation} \label{real targets}
\uno := \left\{                                    
\left(
\begin{matrix}
x   &   -y  \\
S_1^{-1}\, y   &      K^{-1}\,x	
\end{matrix}
\right) \, \colon \, x,y \in \R 
\right\} \,,
\qquad	
\due := \left\{                                    
\left(
\begin{matrix}
x   &   -y  \\
S_2 \, y   &     K \, x	
\end{matrix}
\right) \, \colon \, x,y \in \R 
\right\}	 \,.
\end{equation} 
Let $\sigma \in L^{\infty}(\Omega;\{\sigma_1,\sigma_2\})$. 
It is easy to check  (see for example \cite[Lemma 3.2]{afs}) that $u$ solves \eqref{the pde} if and only if $f$ solves the differential inclusion
\begin{equation} \label{differential inclusion}
\nabla f (x) \in \uno \cup \due  \quad \text{a.e. in } \,\, \Omega \,,
\end{equation}
where $f:=(u,v)$ and $v$ is the stream function of $u$, which is defined, up to an addictive constant, by \eqref{stream}.

In order to solve the differential inclusion \eqref{differential inclusion},
it is convenient to use \eqref{real to conf} and write our target sets in conformal coordinates:
\begin{equation} \label{conf targets}
\uno = \left\{                                    
\left(
a, d_1 (\conj{a})
\right) \, \colon \, a \in \C 
\right\},
\qquad	
\due = \left\{                                    
\left(a, -d_2 (\conj{a})
\right) \, \colon \, a \in \C 
\right\}	,
\end{equation}
where the operators $d_j \colon \C \to \C$ are defined as
\begin{equation} \label{k s}
d_j (a) := k \, \Re a + i \, s_j \, \Im a \,,  \quad \text{with} \quad k:= \frac{K-1}{K+1}  \quad \text{and}  \quad  s_j := \frac{S_j-1}{S_j+1} \,.
\end{equation}
Conditions \eqref{ellipt real} imply
\begin{equation} \label{ellipt conf}
0<k<1  \quad  \text{and}  \quad -k \leq s_j \leq k  \quad \text{for} \quad j=1,2 \,.   
\end{equation}
Introduce the quantities 
\begin{gather} 
s:=\frac{s_1+s_2}{2} = \frac{S_1 S_2 -1}{(1+S_1)(1+S_2)} \label{little s} \\
S := \frac{1+s}{1-s} = \frac{S_1 + S_2 + 2 S_1 S_2}{2 + S_1 + S_2} \label{def S}	\,.
\end{gather}
By \eqref{ellipt conf} we have
\begin{equation} \label{average bounds}
-k \leq s \leq k    \qquad \text{and} \qquad \frac{1}{K} \leq S \leq K \,.   	
\end{equation}
We distinguish three cases.

\smallskip
\paragraph{\indent \textit{1. Case $s > 0$}  (corresponding to $S>1$).} We study this case in Section \ref{sec:convex}, where 
we generalise the methods used in \cite[Section 3.2]{afs}. 
Observe that this case includes the one studied in \cite{afs}. Indeed, 
for $s=k$ one has that $s_1=s_2=k$ and the target sets \eqref{conf targets} become
\[
\uno = E_k= \left\{                                    
\left(
a, k \conj{a}
\right) \, \colon \, a \in \C 
\right\},
\qquad	
\due = E_{-k}= \left\{                                    
\left(a, -k \conj{a}
\right) \, \colon \, a \in \C 
\right\}	,
\]
where $E_{\pm k}$ are defined in \eqref{conformal set}.
We remark that, in this particular case, the construction provided in Section 5 coincides with the one given in \cite[Section 3.2]{afs}.

\smallskip
\paragraph{\indent \textit{2. Case $s < 0$}  (corresponding to $S<1$).} This case can be reduced to the previous one. 
Indeed, if we introduce $\hat{s}_j:= - s_j$, $\hat{s}:=(\hat{s}_1+ \hat{s}_2 )/2 > 0$ and the operators $\hat{d}_j (a) := k \, \Re a + i \, \hat{s}_j \, \Im a$ then the target sets \eqref{conf targets} read as 
\[
\uno = \{                                    
(
a, \hat{d}_1 (a)  ) \, \colon \, a \in \C 
\},
\qquad	
\due = \{                                    
(a, -\hat{d}_2 (a)   ) \, \colon \, a \in \C 
\}	.
\]
This is the same as the previous case, since the absence of the conjugation does not affect the geometric properties relevant to 
the constructions of Section \ref{sec:convex}.

We notice that this case includes $s=-k$ for which the target sets become
 \[
\uno = \left\{                                    
\left(
a, k a
\right) \, \colon \, a \in \C 
\right\} \,,
\qquad	
\due = \left\{                                    
\left(a, -k a
\right) \, \colon \, a \in \C 
\right\}	.
\]
We remark that  in this case, \eqref{differential inclusion} coincides with the classical Beltrami equation (see also \cite[Remark 3.21]{afs}).

\smallskip 
\paragraph{\indent \textit{3. Case $s = 0$} (corresponding to $s_1=-s_2$, $S_1=1/S_2$)} This is a degenerate case, in the sense that the constructions provided in Section 5 for $s>0$ are not well defined. Nonetheless, Theorem \ref{main theorem} still holds true. In fact, as already pointed out in \cite[Section A.3]{npp}, by an affine change of variables, the existence of a solution can be deduced by 
\cite[Lemma 4.1,Theorem 4.14]{afs}, where the authors prove the optimality of the lower critical exponent $\frac{2K}{K+1}$ for 
 the solution of a system in non-divergence form.
We remark that in this case Theorem \ref{main theorem} actually holds in the stronger sense of exact solutions, namely, 
there exists $u \in W^{1,1} (\Omega;\R)$ solution to \eqref{pde3} and such that
\[
\nabla u \in L^{\frac{2K}{K+1}}_{\rm weak} (\Omega; \R^2) \,, \quad \nabla u \notin L^{\frac{2K}{K+1}}(\Omega;\R^2) \,.
\] 
\section{The case $s>0$} \label{sec:convex}
In the present section we prove Theorem \ref{main theorem} under the hypothesis that the average $s$ is positive, namely that
\begin{equation} \label{case2}
\begin{aligned}
& 0<k<1 \,\, \text{ and } \,\,   -s_2< s_1 \leq s_2\,,  \,\, \text{ with } \,\, 0< s_2 \leq k\,, \text{ or }   \\
& 0<k<1 \,\, \text{ and } \,\,   -s_1< s_2 \leq s_1\,,  \,\, \text{ with } \,\, 0< s_1 \leq k\,.  
\end{aligned}
\end{equation}
From \eqref{case2}, recalling definitions \eqref{k s}, \eqref{little s}, \eqref{def S}, we have
\begin{gather} 
0<s \leq k \, , \qquad  1 < S \leq K \,,     \label{average} \\
1/S_2 < S_1 \leq S_2  \,, \quad 1< S_2 \leq K 	 \,,\quad
\text{ or } \quad1/S_1 < S_2 \leq S_1  \,, \quad 1< S_1 \leq K\,.
 \label{K bounds}
\end{gather}
In order to prove Theorem \ref{main theorem}, we will solve the differential inclusion \eqref{differential inclusion} by  
adapting the convex integration program developed in \cite[Section 3.2]{afs} to the present context. 
As already pointed out in the Introduction,
the anisotropy of the coefficients $\sigma_1,\sigma_2$ poses some technical difficulties in the construction of the so-called staircase laminate, needed to
obtain the desired approximate solutions. 
In fact, the anisotropy of $\sigma_1,\sigma_2$ translates into the lack of conformal invariance (in the sense of \eqref{conformal invariance}) of 
the target sets \eqref{conf targets}, while the constructions provided in \cite{afs} heavily rely on the conformal invariance of the target set
$E_{\{-k,k\}}$.
We point out that the lack of conformal invariance was a source of difficulty in \cite{npp} as well, for the proof of the optimality of the upper exponent.

This section is divided as follows. 
In Section \ref{sec:rank one} 
we establish some geometrical properties of rank-one lines in $\matrici$, that will be used in Section \ref{sec:weak} for the construction of the staircase laminate. For every sufficiently small $\delta>0$, such laminate allows us 
to define (in Proposition \ref{prop:grad}) a \textit{piecewise affine} map $f$ that solves the differential inclusion \eqref{differential inclusion} up to an arbitrarily small $L^{\infty}$ error. Moreover $f$ will have the desired integrability properties (see \eqref{crescita}, that is,
\[
\nabla f \in L^p_{\rm weak}(\Omega; \matrici) \,, \quad  p \in \left( \frac{2K}{K+1}-\delta ,\frac{2K}{K+1} \right] \,, \quad \nabla f \notin L^{\frac{2K}{K+1}}(\Omega;\matrici)\,.
\] 
Finally, in Theorem \ref{thm finale}, we remove the 
$L^{\infty}$ error introduced in Proposition \ref{prop:grad}, by means of a standard argument (see, e.g., \cite[Theorem A.2]{npp}).

Throughout this section $\cK >1$ will denote various constants depending on $K,S_1$ and $S_2$, whose precise value may change from place to place.
The complex conjugation is denoted by $J:=(0,1)$ in conformal coordinates, i.e., $Jz= \conj{z}$ for $z \in \C$. 
Moreover, $R_\theta:=(e^{i\theta},0) \in \rotazioni$ denotes the counter clockwise rotation of angle $\theta \in (-\pi,\pi]$.
Define the the argument function 
\[
\arg z := \theta \,, \quad \text{where} \quad z=|z|e^{i \theta} \,, \quad \text{with} \quad \theta \in (-\pi,\pi] \,.
\]
Abusing notation we  write
$\arg R_\theta=\theta$. For $A=(a,b) \in \matrici \setminus \{0\}$ we set
\begin{equation} \label{theta A}
\theta_A :=- \arg (b-d_1 (\conj{a})) \,. 
\end{equation}

\subsection{Properties of rank-one lines} \label{sec:rank one}

In this Section we will establish some geometrical properties of rank-one lines in $\matrici$. 
Lemmas \ref{lemma1}, \ref{lemma2} are generalizations of \cite[Lemmas 3.14, 3.15]{afs} to our target sets \eqref{conf targets}.
In Lemmas \ref{lemma4}, \ref{lemma5} we will study certain rank-one lines connecting $\target$ to $E_\infty$, that will be 
used in Section \ref{sec:weak} to construct the staircase laminate.

\begin{lemma} \label{lemma: proprieta T}
Let $Q \in T_j$ with $j \in \{1,2\}$ and $T_j$ as in \eqref{conf targets}. Then
\begin{gather} 
\det Q > 0  \quad  \text{ for } \quad Q \neq 0 \, ,   \label{positive det} \\
\va{s_j} \leq \va{\mu_Q} \leq k \,, \label{defomation2}\\
\max\{S_j,1/S_j \} \leq K(Q) \leq   K\,.   \label{distortion2} 
\end{gather}
\end{lemma}
\begin{proof}
Let $Q=(q,d_1 (\conj{q})) \in \uno$. By \eqref{ellipt conf} we have 
$|s_1| |q| \leq | d_1(q) | \leq k |q|$ which readily implies \eqref{defomation2} and
\[
(1-k^2) \va{q}^2 \leq \det (Q) \leq (1-s_1^2) \va{q}^2 \,.
\]
The last inequality implies \eqref{positive det}. Finally $K(Q)$ is increasing with respect to $|\mu_{Q}|\in(0,1)$, 
therefore \eqref{distortion2} follows from \eqref{defomation2}.
The proof is analogous if  $Q \in \due$.
\end{proof}

\begin{lemma}\label{lemma1}
Let $A,B \in \matrici$ with $\det B \neq 0$ and $\det (B-A)=0$, then
\begin{equation}
\va{B} \leq \sqrt{2} \, K(B)  \va{A} . 	
\end{equation}
In particular, if $A \in \matrici$ and $Q \in T_j$, $j\in\{1,2\}$, are such that $\det(A - Q) =0$, 
then
\[
\dist (A,T_j) \leq \va{A-Q} \leq (1 + \sqrt{2} K) \dist (A,T_j) \,.
\]
\end{lemma}

\begin{proof}
The first part of the statement is exactly like in \cite[Lemma 3.14]{afs}. For the second part, one can easily adapt the proof of \cite[Lemma 3.14]{afs} 
to the present context taking into account \eqref{positive det} and \eqref{distortion2}. For the reader's convenience we recall the argument. 
Let $A \in \matrici, Q \in \uno$ and $Q_0 \in \uno$ such that $\dist (A,\uno)=|A-Q_0|$. By \eqref{positive det}, 
we can apply the first part of the lemma to $A-Q_0$ and $Q-Q_0$ to get
\[
|Q-Q_0| \leq \sqrt{2} K(Q-Q_0) |A-Q_0| \leq \sqrt{2} K |A-Q_0| \,,
\]
where the last inequality follows from \eqref{distortion2}, since $Q-Q_0 \in \uno$. Therefore
\[
|A-Q| \leq |A-Q_0| + |Q-Q_0| \leq (1+ \sqrt{2} K) |A-Q_0| =  (1+ \sqrt{2} K) \dist(A,\uno) \,. 
\]
The proof for $\due$ is analogous.
\end{proof}

\begin{lemma}\label{lemma2}
Every $A= (a,b) \in \matrici \smallsetminus \{0\}$ lies on a rank-one segment connecting $\uno$ and $E_{\infty}$. Precisely, there exist matrices $Q \in \uno \smallsetminus \{ 0\}$ and
$P \in E_{\infty} \smallsetminus \{ 0\}$, with $\det (P-Q)=0$, such that
$A \in [Q,P]$. We have $P=tJR_{\theta_A}$ for some $t>0$ and $\theta_A$ as in \eqref{theta A}. Moreover, there exists a constant $\cK >1$, depending only on $K,S_1,S_2$, such that
\begin{equation} \label{estimates2}
\frac{1}{\cK} \va{A} \leq \va{P-Q}, \va{P}, \va{Q} \leq \cK \va{A}	\,.
\end{equation}
\end{lemma}

\begin{proof}
The proof can be deduced straightforwardly from the one of \cite[Lemma 3.15]{afs}.  
We decompose any $A=(a,b)$ as 
\[
A= (a,d_1(\conj{a})) + \frac{1}{t} (0, tb- t d_1 (\conj{a})) = Q + \frac{1}{t} P_t\,,
\]	
with $Q \in \uno$ and $P_t \in E_{\infty}$. The matrices $Q$ and $P_t$ are rank-one connected if and only if $\va{a}=\va{d_1 (\conj{a}) + t (b - d_1 (\conj{a}))   }$. Since $\det Q > 0$ for $Q \neq 0$, it is easy to see that
there exists only one $t_0>0$ such that the last identity is satisfied. We then set 
$\rho:=1+1/t_0$ so that
\[
A=\frac{1}{\rho} (\rho \, Q)  +   \frac{1}{t_0 \rho} (\rho \, P_{t_0}) \,.
\]
The latter is the desired decomposition, since $\rho \, Q \in \uno$, $\rho P_{t_0} \in E_{\infty}$ are rank-one connected, $\rho>0$ and $\rho^{-1} + (t_0 \rho)^{-1}=1$. 
Also notice that $\rho P_{t_0} = \rho t_0 |b-d_1(\conj{a})| J R_{\theta_A}$ as stated.

Finally let us prove \eqref{estimates2}. Remark that
\[
\dist (A,\uno) + \dist (A,E_\infty) \leq |A-P| + |A-Q| = |P-Q| \,.
\]
By the linear independence of $\uno$ and $E_{\infty}$,
we get
\[
\frac{1}{c_K} |A| \leq |P-Q| \,. 
\] 
Using Lemma \ref{lemma1}, \eqref{positive det} and \eqref{distortion2} we obtain
\[
|P| \leq c_K |A| , \quad |Q| \leq c_K |A|, \quad |Q| \leq c_K |P| , \quad |P| \leq c_K |Q|.    
\]
By the triangle inequality,
\[
|P-Q| \leq |P| + |Q| \leq (1+c_K) \min ( |P|,|Q| ),
\]
and \eqref{estimates2} follows.
\end{proof}

We now turn our attention to the study of rank-one connections between the target set $\target$ and $E_{\infty}$. 

\begin{lemma} \label{lemma4}
Let $R=(r,0)$ with $\va{r}=1$ and $a \in \C \smallsetminus \{0\}$. For $j \in \{1,2\}$ define
\begin{gather}
Q_1 (a) :=\lambda_1 (a,d_1 (\conj{a}))  \in \uno \,,   \quad Q_2 (a) := \lambda_2 (-a,d_2 (\conj{a})) \in \due \,,  \notag \\
\lambda_j (a) := \frac{1}{ \sqrt{B_j^2 (a) + A_j (a)} + B_j (a)}	 \,, \label{lambda} \\
\begin{cases}
A_j (a) := \det (a,d_j (a)) = \va{a}^2 - \va{d_j (a)}^2 \,, \\
B_j (a) := \Re \, (\conj{r} \, d_j (a)) \,.
\end{cases}   \label{coefficienti}
\end{gather}
Then $\lambda_j >0$, $A_j >0$ and $\det(Q_j-JR)=0$. Moreover there exists a constant $\cK>1$ depending only on $K,S_1,S_2$ such that
\begin{equation} \label{stima Q}
\frac{1}{\cK} \leq \va{Q_j (a)} \leq \cK \,,	
\end{equation}
 for every $a \in \C \smallsetminus \{0\}$ and $R \in \rotazioni$.
\end{lemma}

\begin{proof}
Condition $\det(Q_j-JR)=0$ is equivalent to 
$|\lambda_j a |= | \lambda_j d_j (\conj{a}) - \conj{r} |$, that is
\begin{equation} \label{secondo grado}
A_j (a)  \lambda_j^2 + 2 B_j (a) \lambda_j  - 1 = 0
\end{equation}
with $A_j,B_j$ defined by \eqref{coefficienti}. Notice that $A_j >0$ by \eqref{positive det}. Therefore $\lambda_j$ defined in \eqref{lambda} 
solves \eqref{secondo grado} and satisfies $\lambda_j >0$. 

We will now prove \eqref{stima Q}. Since $a \neq 0$, we can write $a= t \omega$ for some $t>0$ and $\omega \in \C$, with $\va{\omega}=1$. We have
$A_j(a)= t^2 A_j (\omega)$ and $B_j(a)= t B_j (\omega)$ 
so that $\lambda_j(a) = \lambda_j (\omega)/t$.
Hence
\begin{equation} \label{eq Qi}
Q_1 (a) = \lambda_1 (\omega) (\omega,d_1(\conj{\omega})) \,, \quad  Q_2 (a)  = \lambda_2 (\omega) (-\omega,d_2(\conj{\omega}))\,.
\end{equation}
Since $\lambda_j$ is continuous and positive in $(\C \smallsetminus \{0\}) \times \rotazioni$, \eqref{stima Q} follows from \eqref{eq Qi}. 
\end{proof}

\smallskip
\textbf{Notation.}
Let $\theta \in (-\pi,\pi]$. For $R_\theta=(e^{i\theta},0) \in \rotazioni$, define $x:= \cos \theta, y:= \sin \theta$ and
\begin{equation} \label{scelta di a}
a(R_\theta):= \frac{x}{k} + i \, \frac{y}{s} \,, 
\end{equation}
where $s$ is defined in \eqref{little s}. 
Identifying $SO(2)$ with the interval $(-\pi,\pi]$,
for $j=1,2$, we introduce the function
\begin{equation} \label{def lambda i}
\lambda_j \colon (-\pi,\pi] \to (0,+\infty) \qquad \text{defined by} \qquad \lambda_j(R_\theta):=\lambda_j (a(R_\theta)) 
\end{equation}
with $\lambda_j (a(R_\theta))$ as in \eqref{lambda}.
Furthermore, for $n \in \N$ set
\begin{equation} \label{L h}
\begin{gathered} 
M_j (R_\theta) := \frac{\lambda_j }{\displaystyle \frac{\lambda_1  + \lambda_2}{2} - \lambda_1 \lambda_2} \,, \quad l(R_\theta):=  \frac{M_1 + M_2}{2} -1  \,, \quad m := \min_{ \theta \in (-\pi,\pi]} \frac{ M_2}{2-M_2}    \\
L(R_\theta):= \frac{1+l}{1-l}  \, , \quad \beta_n (R_\theta) := 1 - \frac{1+l}{n} \,,  \quad  p (R_{\theta}):= \frac{2L}{L+1} \,.  
\end{gathered}
\end{equation}

\begin{lemma} \label{lemma5}
For $j =1,2$, the functions
\begin{gather*}
\lambda_j \colon (-\pi,\pi] \to \left[  \frac{s}{1+s_j},\frac{k}{1+k}    \right] \,, \qquad  
l \colon (-\pi ,\pi] \to [s,k] \,, \\  
L \colon (-\pi,\pi] \to \left[  S,K \right]  \,, \qquad   p \colon (-\pi,\pi] \to \left[ \frac{2S}{S+1},\frac{2K}{K+1} \right] \,,
\end{gather*}
are even, surjective and their periodic extension is $C^1$. Furthermore,   
they are strictly decreasing in $(0,\pi/2)$ and strictly increasing in $(\pi/2,\pi)$, with maximum at $\theta=0,\pi$ and minimum at $\theta=\pi/2$.
Finally 
\begin{gather}
0<M_j <2  \, , \qquad 	m>0  \label{stima M} \,, \\
 \prod_{j=1}^n  \beta_j (R_{\theta}) = \frac{1}{n^{p(R_\theta)}} + O\left(\frac{1}{n}\right)  \,, \label{asintotica produttoria}
\end{gather}
 where $O(1/n) \to 0$ as $n \to \infty$ uniformly for $\theta \in (-\pi,\pi]$.
\end{lemma}

\begin{proof}
Let us consider $\lambda_j$ first. By definitions \eqref{coefficienti}, \eqref{scelta di a} and by recalling that $x^2+y^2 = 1$, 
we may regard $A_j,B_j$ and $\lambda_j$ as functions of $x \in [-1,1]$. 
In particular,
\begin{equation} \label{coefficienti2}
A_j (x)  =  \left( \frac{1-k^2}{k^2}  -  \frac{1-s_j^2}{s^2}   \right) x^2 +    \frac{1-s_j^2}{s^2}  \, , \quad
B_j (x)   = \left( 1 - \frac{s_j}{s} \right)  x^2 +\frac{s_j}{s} \,.
\end{equation}
By symmetry we can restrict to $x \in [0,1]$.
We have three cases:

\smallskip
\textit{1. Case $s_1=s_2$.} Since $s_1=s_2=s$, from \eqref{coefficienti2} we compute 
\[
\lambda_1 (x) =\lambda_2(x) = { \left( 1 + \sqrt{  \left( \frac{1}{k^2} - \frac{1}{s^2}    \right) x^2 +  \frac{1}{s^2}  \,  } \right) }^{-1} \,.
\]
By \eqref{case2},\eqref{average} this is a strictly increasing function in $[0,1]$, and the rest of the thesis for $\lambda_j$ readily follows.

\smallskip
\textit{2. Case $s_1<s_2$.}
By \eqref{case2} we have 
\begin{equation} \label{second case}
-s_2<s_1<s  \qquad \text{and} \qquad 0<s<s_2 \,.
\end{equation}
Relations \eqref{coefficienti2} and \eqref{second case} imply that
\begin{align}
A_j' (0)=0 \,, \quad A_j' (x) < 0 \,, \quad \text{ for } \quad x \in (0,1] \,,    \label{der Aj} 	\\
B_1' (0)=0 \,, \quad  B_1' (x) > 0 \,, \quad \text{ for } \quad x \in (0,1]  \,,\label{der B1} \\
B_2' (0)=0 \,, \quad   B_2' (x) < 0 \,, \quad \text{ for } \quad x \in (0,1]  \,.\label{der B2} 
\end{align}
We claim that
\begin{equation} \label{claim1}
	\lambda_j' (0)=0  \,, \quad \lambda_j' (x) > 0 \,, \quad \text{ for } \quad x \in (0,1] \,.
\end{equation}
Before proving \eqref{claim1}, notice that $\lambda_j(0)= \displaystyle\frac{s}{1+s_j}$ and $\lambda_j(1)=\displaystyle \frac{k}{1+k}$, therefore the surjectivity of $\lambda_j$ will follow from \eqref{claim1}. 
Let us now prove \eqref{claim1}.
For $j=2$ condition \eqref{claim1} is an immediate consequence of the definition of $\lambda_2$ and \eqref{der Aj}, \eqref{der B2}. For $j=1$ we have
\begin{equation} \label{der lambda1}
\lambda_1'(x) = -\frac{1}{\lambda_1^2} \left(    \frac{A_1'+2 B_1 B_1'}{2 \sqrt{B_1^2 + A_1}} + B_1 ' \right) 
\end{equation}
and we immediately see that $\lambda_1'(0)=0$ by \eqref{der Aj} and \eqref{der B1}. Assume now that $x \in (0,1]$. By \eqref{der B1} and \eqref{der lambda1}, the claim \eqref{claim1} is equivalent to
\[
{A_1'}^2 + 4 A_1' B_1 B_1' - 4 A_1 {B_1'}^2 > 0 \,,   \quad \text{ for } \quad x \in (0,1] \,.
\] 
After simplifications, the above inequality is equivalent to
\begin{equation} \label{cond f}
 \frac{4 f(s_1,s_2)}{k^4 {(s_1+s_2)}^4} \, x^2 >0 \,,     \quad \text{ for } \quad x \in (0,1]\,,
\end{equation}
where $f(s_1,s_2)=a b c d$, with
\begin{align*}
a   = -2k + (1+k) s_1 + (1-k) s_2   \,, \qquad 
b   = 2k + (1+k) s_1 + (1-k) s_2 \,,\\	
c   = -2k - (1-k) s_1 - (1+k) s_2  \,, \qquad
d   = 2k - (1-k) s_1 - (1+k) s_2  \,.
\end{align*}
We have that $a,c <0$ since $s_1<s_2$ and $b,d>0$ since $s_1>-s_2$. Hence
\eqref{cond f} follows.

\smallskip
\textit{3. Case $s_2<s_1$.} In particular we have
\begin{equation} \label{third case}
-s_1<s_2<s  \qquad \text{and} \qquad 0<s<s_1 \,.
\end{equation}
This is similar to the previous case. Indeed \eqref{der Aj} is still true, but for $B_j$ we have
\begin{align}
B_1' (0)=0 \,, \quad  B_1' (x) < 0 \,, \quad \text{ for } \quad x \in (0,1] \,,  \label{der B1 due} \\
B_2' (0)=0  \,, \quad B_2' (x) > 0 \,, \quad \text{ for } \quad x \in (0,1] \,.  \label{der B2 due} 
\end{align}
This implies \eqref{claim1} with $j=1$. Similarly to the previous case, we can see that \eqref{claim1} for $j=2$ is equivalent to 
\begin{equation} \label{cond f due}
 \frac{4 f(s_2,s_1)}{k^4 {(s_1+s_2)}^4} \, x^2 >0 \,,    \quad \text{ for } \quad x \in (0,1] \,.
\end{equation}
Notice that $f$ is symmetric, therefore \eqref{cond f due} is a consequence of \eqref{cond f}.

\smallskip
We will now turn our attention to the function $l$. Notice that
\begin{equation} \label{l identity}
l= \frac{1}{1-H} - 1  \,, \quad \text{where} \quad
H:=\frac{2 \lambda_1 \lambda_2}{ \lambda_1 + \lambda_2} = 2 \, {\left( \frac{1}{\lambda_1} + \frac{1}{\lambda_2}   \right)}^{-1}
\end{equation}
is the harmonic mean of $\lambda_1$ and $\lambda_2$. 
Therefore $H$ is differentiable and even. 
By direct computation we have 
\[
H'=2 \, \frac{\lambda_1' \lambda_2^2+ \lambda_1^2 \lambda_2'}{   {(\lambda_1+\lambda_2)}^2  } \,. 
\]
Since $\lambda_j >0$, by \eqref{claim1} we have
\begin{equation} \label{derivata H}
	H' (0)=0 \,, \quad  H' (x) > 0 \,, \quad \text{ for } \quad x \in (0,1]  \,.
\end{equation}
Moreover $H(0)=\displaystyle\frac{s}{1+s}$ and $H(1)=\displaystyle\frac{k}{1+k}$. Then from \eqref{l identity}
we deduce $l(0)=s, l(1)=k$ and the rest of the statement for $l$.

\smallskip
The statements for $L$ and $p$ follow directly from the properties of $l$ and from the fact that 
$t \to \displaystyle \frac{1+t}{1-t}$, $t \to \displaystyle \frac{2t}{t+1}$ are $C^1$ and strictly increasing for $0<t<1$ and $t>1$, respectively. 

\smallskip
Next we prove \eqref{stima M}. By \eqref{case2} and the properties of $\lambda_j$, we have in particular  
\begin{equation} \label{stima lambda bis}
0< \lambda_j < \frac{1}{2} \,, \quad 
0<H < \frac{1}{2} 	\,,
\end{equation}
where $H$ is defined in \eqref{l identity}. 
Since $\lambda_j>0$, the inequality $M_j>0$ is equivalent to $H<1$, which holds by \eqref{stima lambda bis}. The inequality $M_2<2$ is instead equivalent to 
$\lambda_1 (1-2 \lambda_2)>0$, which is again true by \eqref{stima lambda bis}. The case $M_1<2$ is similar. Finally $m>0$ follows from $0<M_2<2$ and the continuity of $\lambda_j$.

\smallskip
Finally we prove \eqref{asintotica produttoria}. By definition we have $1+l= \displaystyle \frac{2 L}{L+1}= p$.  
By taking the logarithm of $\prod_{j=1}^n \beta_j (R_\theta)$, we see that there exists a constant $c>0$, depending only on $K,S_1,S_2$, such that
\begin{equation} \label{prod beta}
\va{  \log \left(   \prod_{j=1}^n  \beta_j (R_{\theta})   \right) +   p(R_\theta) \log n   }< c \,, \quad \text{for every} \quad \theta \in (-\pi,\pi] \,.
\end{equation}
Estimate \eqref{prod beta} is uniform because $\beta_j$ and $p$ are $\pi$-periodic and uniformly continuous. 
\end{proof}

\subsection{Weak staircase laminate} \label{sec:weak}

We are now ready to construct a staircase laminate in the same fashion as \cite[Lemma 3.17]{afs}.  
The steps of our staircase will be the sets
\[
\step_n := n J \rotazioni = \left\{ (0,n e^{i \theta}) \, \colon \, \theta \in (-\pi,\pi] \right\} \,, \quad n \geq 1 \,.
\]
For $0< \delta < \pi/2$ we introduce the set
\[
E_{\infty}^\delta := \{ (0,z) \in E_\infty \, \colon \, |\arg z|< \delta \} \,, \qquad \step_n^\delta := \step_n \cap E_\infty^\delta  \,.
\]

\begin{lemma}\label{lemma3}  
Let $0<\delta<\pi/4$ and $0<\rho<\min \{m, \frac{1}{2}\}$, with $m>0$ defined in \eqref{L h}.
There exists a constant $\cK >1$ depending only on $K,S_1, S_2$, such that for every $A=(a,b) \in \matrici$ satisfying 
\begin{equation} \label{dist}
\dist(A,\step_n) < \rho	\,,
\end{equation}
there exists a laminate of third order $\nu_A$, such that:
\begin{enumerate}[(i)]
\item	$\conj{\nu}_A=A$, \smallskip
\item $\spt \nu_A  \subset \target \cup \step_{n +1}\,,$ \smallskip
\item $\spt \nu_A  \subset \{ \xi \in \matrici \, \colon \, \cK^{-1} n < \va{\xi} < \cK \, n \} \,, $  \smallskip
\item $\spt \nu_A \cap \step_{n+1} = \{ (n+1) J R \}$, with $R=R_{\theta_A}$ as in \eqref{theta A}.
\end{enumerate}
Moreover  
\begin{equation} \label{growth}
\left( 1 - \cK \, \frac{\rho}{n}  \right) \beta_{n} (R) \leq \nu_A (\step_{n +1}) \leq \left( 1 + \cK \, \frac{\rho}{n}  \right) \beta_{n+2} (R) \,, 	
\end{equation}
where $\beta_n$ is defined in \eqref{L h}. If in addition $n \geq 2$ and
\begin{equation} \label{dist angolo}
\dist (A , \step_n^\delta) < \rho \,,
\end{equation}
then 
\begin{equation} 
|\arg{R}|=|\theta_A | < \delta + \rho \,.  \label{arg R bis}
\end{equation} 
In particular $\spt \nu_{A} \subset T \cup \step_{n+1}^{\delta+\rho}$.
\end{lemma}

\begin{figure}[t!]
\centering   
\def\svgwidth{8.5cm}   
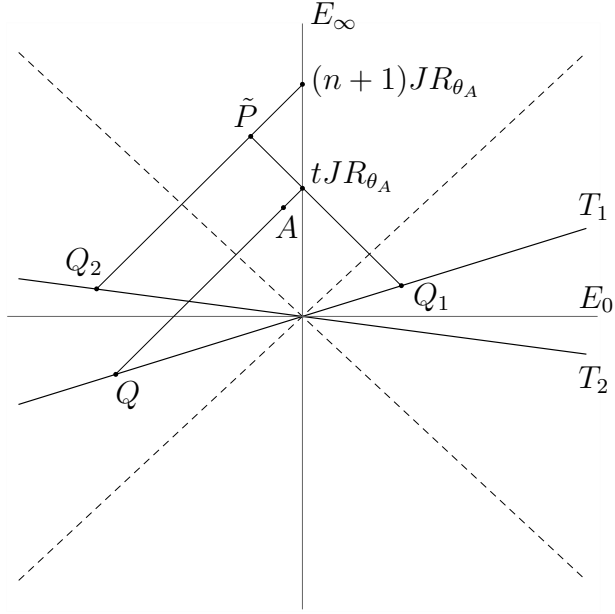  
\caption{Weak staircase laminate.}   
\label{fig:staircase}
\end{figure}

\begin{proof}
Let us start by defining $\nu_A$. 
From Lemma \ref{lemma2} there exist $c_K >1$ and non zero matrices $Q \in \uno$, $P \in E_{\infty}$, such that $\det(P-Q)=0$,
\begin{gather} 
A=\mu_1 Q + (1-\mu_1) P \,, \quad \text{for some } \quad \mu_1 \in [0,1]\,, \label{split} \\ 
\frac{1}{\cK} \va{A} \leq \va{P-Q}, \va{P}, \va{Q} \leq \cK \va{A} \,.  \label{estimates2 bis}
\end{gather}
Moreover $P=tJR$ with $R=R_{\theta_A}=(r,0)$ as in \eqref{theta A} and $t>0$. 
 We will  estimate $t$. 
 By \eqref{dist}, there exists $\tilde{R} \in \rotazioni$ such that $|A-n J \tilde{R} |< \rho$. Applying Lemma \ref{lemma1} to $A-n J \tilde{R}$ and $P - n J \tilde{R}$ yields
\begin{equation} \label{stima segmento}
|P- n J \tilde{R}|< \sqrt{2} \rho \,,	
\end{equation}
since $P- n J \tilde{R} \in E_{\infty}$. Hence from \eqref{stima segmento} we get 
\begin{equation} \label{stima t}
\va{t - n} < \rho	\,,
\end{equation}
since $|JR|=|J \tilde{R}|= \sqrt{2}$. We also have 
\begin{equation}\label{mu1}
\mu_1 = \frac{\va{A-Q}}{\va{P-Q}} \geq 1 - \frac{\va{P-A}}{\va{P-Q}}\geq 1 - \cK \frac{\rho}{n} \,,
\end{equation}
since $\va{P-A}<3 \rho$ and $\va{P-Q} > n/\cK$, by \eqref{dist angolo}, \eqref{estimates2 bis},  \eqref{stima segmento}.

Next we split $P$ in order to  ``climb'' one step of the staircase (see Figure \ref{fig:staircase}). 
Define $x:=\cos \theta_A,y:=\sin \theta_A$ and
\[
a:= \frac{x}{k} + i \, \frac{y}{s} \,,
\]
as in \eqref{scelta di a}. Moreover set
\[
Q_1 := \lambda_1 (a,d_1 (\conj{a})) \,, \quad  Q_2 := \lambda_2 (-a,d_2 (\conj{a})) \,.
\]
Here $\lambda_1,\lambda_2$ are chosen as in \eqref{lambda}, so that $Q_j \in \target_j$ and, by   Lemma \ref{lemma4}, $\det(Q_j - JR)=0$. 
Furthermore, set
\begin{equation} \label{comb convessa}
\left\{
\begin{aligned}
& \mu_2 := \frac{ M_2 - (t - n) M_2 }{2 n +  M_2 + (t - n)(2 - M_2)} \,, \\	
& \mu_3 := \frac{M_1 - (t-n) M_1}{2 (n + 1)} \,,
\end{aligned}
\right.
\end{equation}
with $M_j$ as in \eqref{L h}. With the above choices we have
\begin{equation} \label{staircase}
\left\{
\begin{aligned}
& t J R = \mu_2 t Q_1 + (1-\mu_2) \tilde{P} \,, \\
& \tilde{P}= \mu_3 (n+1) Q_2 + (1-\mu_3) (n+1) JR \,,	
\end{aligned}
\right.
\end{equation}
and $\mu_2,\mu_3 \in [0,1]$ by \eqref{stima M}. 
In order to check \eqref{staircase}, we solve the first equation in $\tilde{P}$ to get
\begin{equation} \label{equiv}
\gamma_2 t JR + (1-\gamma_2) t Q_1 = \gamma_3 (n + 1) Q_2 + (1-\gamma_3)(n + 1)JR \,,
\end{equation}
with $\mu_2 = 1 - 1/\gamma_2$ and $\mu_3 = \gamma_3$. 
Equating the first conformal coordinate of both sides of \eqref{equiv} yields 
\begin{equation} \label{gamma2}
\gamma_2 = 1 + \gamma_3 \, \frac{n + 1}{t} \frac{\lambda_2}{\lambda_1} \,.
\end{equation}
Substituting \eqref{gamma2} in the second component of \eqref{equiv} gives us
\begin{equation} \label{gamma3}
\gamma_3 \left( \lambda_1 + \lambda_2  - \lambda_1 \lambda_2 \left( d_1 (a) + d_2(a)  \right) \, r^{-1} \right) = \frac{1 - (t-n)}{n + 1 } \, \lambda_1 \,.
\end{equation}
By \eqref{scelta di a}, $d_1(a)+d_2(a)=2 r $ and equation \eqref{gamma3} yields
\begin{equation} \label{gamma3 bis}
\gamma_3 = \frac{1 - (t-n)}{n + 1 } \, \frac{\lambda_1}{ \lambda_1+\lambda_2-2 \lambda_1 \lambda_2} = \frac{1 - (t-n)}{2(n + 1) } \, M_1 \,.
\end{equation}
Equations \eqref{gamma2} and \eqref{gamma3 bis} give us \eqref{comb convessa}.
Therefore, by \eqref{split} and \eqref{staircase}, the measure
\[
\nu_A := \mu_1 \delta_Q +  (1-\mu_1) \left( \mu_2 \delta_{t Q_1} + (1-\mu_2) \left( 
\mu_3 \delta_{(n+1) Q_2} + (1-\mu_3) \delta_{(n+1)JR} \right) \right) 
\]
defines a laminate of third order with barycenter $A$, supported in $\uno \cup \due \cup \step_{n +1}$ and such that
$\spt \nu_A \cap \step_{n+1}= \{(n+1)JR\}$ with $R=R_{\theta_A}$. Moreover
\[
\spt \nu_A \subset  \{ \xi \in \matrici \, \colon \, \cK^{-1} n < \va{\xi} < \cK \, n \} \,,
\] 
since $\cK^{-1} n < |Q|<\cK n$ by \eqref{dist},\eqref{estimates2 bis} and 
$$
\cK^{-1} n <|t Q_1|, |(n+1)Q_2|<\cK n
$$ 
by \eqref{stima t}, \eqref{stima Q}. 
Next we prove \eqref{growth} by estimating
\begin{equation} \label{to show}
\nu_A ( \step_{n +1}) = \mu_1 (1-\mu_2)(1-\mu_3) \,.
\end{equation}
Notice that $\nu_A ( \step_{n +1})$ depends on $R$.
For small $\rho$, we have 
\[
\mu_2 = \displaystyle\frac{M_2}{2 n} + \rho \, O \left(\frac{1}{n}	\right)\,,  \quad   \mu_3=\displaystyle \frac{M_1}{2 n} + \rho \, O \left(\frac{1}{n}\right) \,,
\]
so that
\[
(1-\mu_2) (1-\mu_3)=  1 - \frac{M_1+M_2}{2 n} + \rho \, O \left( \frac{1}{n^2} \right) = 1 - \frac{1+l}{n}  + \rho \, O \left( \frac{1}{n^2} \right)\,,
\]
with $l$ as in \eqref{L h}. Although this gives the right asymptotic, we will need to estimate \eqref{to show} for every $n \in \N$. By direct calculation
\[
(1-\mu_2)(1-\mu_3) = \frac{n + (t- n)}{n + 1} \, \frac{2 n + 2 -  M_1 + (t-n)M_1}{2 n +  M_2 + (t-n) (2-M_2)} \,,
\]
so that
\begin{equation} \label{to show 2}
(1-\mu_2)(1-\mu_3) = 
\left(   1   + \frac{t-n}{n}    \right)
\left(   1   - \frac{1}{n + 1}    \right)	
\left(   1   -  \frac{ 2 l \, (1- (t-n)  )     }{2 n + M_2 + (t-n) (2-M_2)}    \right)\,.
\end{equation}
Let us bound \eqref{to show 2} from above. Recall that $t-n < \rho <1$ and 
$2-M_2 >0$, by \eqref{stima M}, so 
the denominator of the third factor in \eqref{to show 2} is bounded from above by $2 (n+1)$ and
\begin{equation} \label{upper1}
\begin{aligned}
	(1-\mu_2)(1-\mu_3) & \leq  
	\left(   1   + \frac{\rho}{n}    \right)
	\left(   1   - \frac{1}{n + 1}    \right)	
\left(   1   -  \frac{l}{n + 1} + l \, \frac{\rho}{n+1} \right) \\
& \leq \left(   1   + \cK \, \frac{\rho}{n}    \right) 
\left(   1   - \frac{1}{n + 1}    \right)
\left(   1   -  \frac{l}{n + 1} \right) \,,
\end{aligned}
\end{equation}
where $\cK >1$ is such that
\[
l \, \frac{\rho}{n + 1} \left( 1+ \frac{\rho}{n}    \right)  \leq (\cK -1) \, \frac{\rho}{n}  \left( 1 - \frac{l}{n + 1} \right) \,.
\]
Moreover
\begin{equation} \label{upper2}
\left(   1   - \frac{1}{n + 1}    \right)
\left(   1   -  \frac{l}{n + 1} \right) = 1 - \frac{1+l}{n+1} +  
\frac{l}{{(n + 1)}^2} \leq  1 - \frac{1+l}{n+2} = \beta_{n+2} (R) \,.
\end{equation}
The upper bound in \eqref{growth} follows from \eqref{upper1} and \eqref{upper2}.

Let us now bound \eqref{to show 2} from below. We can estimate from below  the denominator in the third factor of \eqref{to show 2} with $2n$,
since $t - n > - \rho$ by \eqref{stima t} and the assumption that $\rho < m$ with $m$ as in \eqref{L h}. 
Therefore
\begin{equation} \label{lower1}
\begin{aligned}
	(1-\mu_2)(1-\mu_3) & \geq  
	\left(   1   - \frac{\rho}{n}    \right)
	\left(   1   - \frac{1}{n + 1}    \right)	
\left(   1   -  \frac{l}{n} - l \, \frac{\rho}{n} \right) \\
& \geq \left(   1   - \cK \, \frac{\rho}{n}    \right) 
\left(   1   - \frac{1}{n + 1}    \right)
\left(   1   -  \frac{l }{n } \right)\,,
\end{aligned}
\end{equation}
if we choose $\cK >1$ such that
\[
\left( 1- \frac{\rho}{n}    \right) \, l \leq (\cK -1) \left( 1 - \frac{l}{n } \right) \,.
\]
Finally
\begin{equation} \label{lower2}
\left(   1   -  \frac{1}{n +1} \right)\left(   1   -  \frac{l}{n } \right) \geq 1 - \frac{1+l}{n} = \beta_n (R) \,.
\end{equation}
The lower bound in \eqref{growth} follows from \eqref{lower1} and \eqref{lower2}.

Finally, the last part of the statement follows from a simple geometrical argument, recalling that
$\arg R =\theta_A= - \arg(b-d_1 (\conj{a}))$ and using hypothesis \eqref{dist angolo}.
\end{proof}

\begin{remark} \label{iteration}
By iteratively applying  Lemma \ref{lemma3}, one can obtain, 
for every $R_\theta \in \rotazioni$, 
a sequence of laminates of finite order $\nu_n \in \lam$ that satisfies 
$\conj{\nu}_n=JR_\theta$, $\spt \nu_n  \subset \uno \cup \due \cup \step_{n +1}$, and  
\begin{equation} \label{explosion}
\lim_{n \to \infty} \int_{\matrici} {|\lambda	|}^{p (R_\theta)} \, d \nu_n (\lambda) = \infty \,,
\end{equation}
where $p(R_\theta) \in \left[ \frac{2S}{S+1}, \frac{2K}{K+1} \right]$ is the function defined in \eqref{L h}.      
Indeed,
setting $A=J R_\theta$ and iterating the construction of Lemma \ref{lemma3}, yields 
$\nu_n \in \lam$ such that $\conj{\nu}_n=JR_\theta$ and $\spt \nu_n  \subset \uno \cup \due \cup \step_{n +1}$. Notice that $\nu_n$ contains the term $\prod_{j=1}^n (1-\mu_2^j)(1-\mu_3^j) \delta_{(n+1) J R_{\theta}}$, with $\mu_2^j,\mu_3^j$ as defined in \eqref{comb convessa}.
Therefore, using \eqref{asintotica produttoria} and \eqref{growth} (with $\rho=0$), we obtain
\begin{equation} \label{explosion2}
\prod_{j=1}^n (1-\mu_2^j)(1-\mu_3^j) \approx \prod_{j=1}^n \beta_j (R_\theta) \approx \frac{1}{n^{p(R_\theta)}} 
\end{equation}
which implies \eqref{explosion}. 
\end{remark}

\begin{remark} \label{rmk:differenza}  
In the isotropic case $S=K$, the laminate $\nu_A$ provided by Lemma \ref{lemma3} coincides with the one in \cite[Lemma 3.16]{afs}. In particular, the growth condition \eqref{growth} is independent of the initial point $A$, and it reads as
\[
\left( 1 - \cK \, \frac{\rho}{n}  \right) \beta_n (I) \leq \nu_A (\step_{n +1}) \leq \left( 1 + \cK \, \frac{\rho}{n}  \right) \beta_{n+2} (I) \,, \quad \beta_n (I)=1- \frac{1+k}{n} \,. 
\]
Moreover, by Remark \ref{iteration}, for every $R_\theta \in SO(2)$, $J R_\theta$ is the center of mass of a sequence of laminates of finite order such that \eqref{explosion} holds with $p(R_\theta) \equiv \frac{2K}{K+1}$, which gives the desired growth rate.

In contrast, in the anisotropic case $1<S<K$, the growth rate of the laminates explicitly depends on the argument of the barycenter $J R_\theta$. The desired growth rate corresponds to $\theta = 0$, that is, the center of mass has to be $J$. 

In constructing approximate solutions with the desired integrability properties, it is then crucial to be able to select rotations whose angle 
whose angle lies in an arbitrarily small neighbourhood of $\theta = 0$.
\end{remark}

We now proceed to show the existence of a \textit{piecewise affine} map $f$ that solves the differential inclusion \eqref{differential inclusion} up to an arbitrarily small $L^{\infty}$ error. Such map will have the integrability properties given by \eqref{crescita}.

\begin{proposition} \label{prop:grad}
Let $\Omega \subset \R^2$ be an open bounded domain.
Let $K>1$, $\alpha \in (0,1)$, $\varepsilon>0$, $0<\delta_0 < \frac{2K}{K+1} - \frac{2S}{S+1}$, $\gamma > 0$. There exist a constant $c_{K,\delta_0} > 1 $, depending only on $K,S_1,S_2,\delta_0$, and a 	\textit{piecewise affine} map $f \in W^{1,1} (\Omega;\R^2) \cap C^{\alpha} (\overline{\Omega};\R^2)$, such that
\begin{enumerate}[\indent(i)]
\item	$f (x)= J x$ on $\partial \Omega$, \smallskip
\item  $[f - J x]_{C^{\alpha} (\overline{\Omega})} < \varepsilon$, \smallskip
\item $\dist (\nabla f (x), \target ) < \gamma$ a.e. in $\Omega$. \smallskip
\end{enumerate}
Moreover
\begin{equation} \label{crescita}
\frac{1}{c_{K,\delta_0}} t^{-\frac{2K}{K+1}} < \frac{   | \{ x \in \Omega \, \colon \, |\nabla f (x)|>t  \} | }{|\Omega|} < c_{K,\delta_0} \, t^{-p} \,,
\end{equation}
where $p \in \left(\frac{2K}{K+1}- \delta_0,\frac{2K}{K+1} \right]$. That is, 
$\nabla f \in L^{p}_{\rm weak} (\Omega;\matrici)$ and $\nabla f \notin L^{\frac{2K}{K+1}} (\Omega;\matrici)$.
In particular
$f \in W^{1,q} (\Omega;\R^2)$ for every $q < p$, but $\int_{\Omega} \va{\nabla f (x)}^{\frac{2K}{K+1}} \, dx = \infty$. 
\end{proposition}

\begin{proof}
By Lemma \ref{lemma5} the function $p \colon (-\pi,\pi] \to \left[ \frac{2S}{S+1},\frac{2K}{K+1}   \right]$ is uniformly continuous. Let $\alpha \colon [0,\infty] \to [0,\infty]$ be its modulus of continuity.
Fix $0<\delta<\pi/4$ such that
\begin{equation} \label{delta massimo}
\alpha(\delta) < \delta_0 \,. 	
\end{equation}
Let $\{\rho_n\}$ be a strictly decreasing positive sequence satisfying
\begin{equation}\label{rho 1}
\rho_1 < \frac{1}{4} \min \{ m,c_K^{-1}, \dist(\step_1,T),\gamma \} \,, \quad  \rho_n <\frac{\delta}{4} \, 2^{-n} \,, 
\end{equation}
where $m>0$ and $c_K >1$ are the constants from Lemma \ref{lemma3}. Define $\{\delta_n\}$ as
\begin{equation} \label{delta n}
\delta_1 := 0  \, \quad \text{and} \quad  \delta_n := \sum_{j=1}^{n-1} \rho_n   \,\, \text{ for } \, n \geq 2 \,. 	
\end{equation}
In particular from \eqref{rho 1},\eqref{delta n} it follows that
\begin{equation} \label{bound delta n}
\delta_ n < \frac{\delta}{2} \,, \quad \text{ for every } \,\, n \in \N\,.	
\end{equation}

\medskip
\paragraph{\textbf{Step 1.}} Similarly to the proof of 
\cite[Proposition 3.17]{afs}, by repeatedly combining Lemma \ref{lemma3} and Proposition \ref{gradienti},  
we will prove the following statement:

\medskip
\paragraph{\textit{Claim.}}
There exist sequences of piecewise constant functions $\tau_n \colon \Omega \to (0 , \infty)$ and piecewise affine
 Lipschitz mappings $f_n \colon \Omega \to \R^2$, such that
\begin{enumerate}[(a)] 
\item \label{prop a} $f_n(x)=Jx$ on $\partial \Omega$, \smallskip 
\item \label{prop b}$[f_n - Jx]_{C^{\alpha} (\overline{\Omega})}  < (1-2^{-n}) \varepsilon$, \smallskip
\item \label{prop c}$\dist(\nabla f_n(x), T \cup \step_n^{\delta_n}) < \tau_n (x)$ a.e. in $\Omega$,   \smallskip
\item \label{prop d}$\tau_n (x) = \rho_n$   in $\Omega_n$, \smallskip	
\end{enumerate}
where
\[
\Omega_n := \{ x \in \Omega \, \colon \, \dist(\nabla f_n(x), T) \geq \rho_n \} \,.
\]
Moreover
\begin{equation}\label{stima omega n}
\prod_{j=1}^{n-1} \left(  1 - c_K   \frac{\rho_j}{j} \right) \beta_j (R_0) \leq
\frac{|\Omega_n|}{|\Omega|} \leq  \prod_{j=1}^{n-1} \left(  1 + c_K   \frac{\rho_j}{j} \right) \beta_{j+2} (R_{\delta})	 \,.
\end{equation}

\medskip
\paragraph{\textit{Proof of the claim.}}
We proceed by induction. 
Set $f_1(x):=Jx$ and $\tau_1 (x) := \rho_1$ for every $x \in \Omega$. 
Since $J \in \step_1^0$, then $f_1$ satisfies \ref{prop a}-\ref{prop c}. Also, 
$\rho_1 < \dist (T, \step_1)/4$ by \eqref{rho 1}, so $\Omega_1 = \Omega$ and \ref{prop d}, \eqref{stima omega n} follow.

Assume now that $f_n$ and $\tau_n$ satisfy the inductive hypothesis. 
We will first define $f_{n+1}$ by modifying $f_n$ on the set $\Omega_n$. Since $f_n$ is piecewise affine we have a decomposition of $\Omega_n$ into pairwise disjoint open subsets $\Omega_{n,i}$ such that
\begin{equation} \label{unione disgiunta}
\va{ \Omega_n \smallsetminus \bigcup_{i=1}^{\infty} \Omega_{n,i}    } = 0 \,,
\end{equation}
with $f_n (x) = A_i x + b_i$ in $\Omega_{n,i}$, for some $A_i \in \matrici$ and $b_i \in \R^2$. Moreover
\begin{equation} \label{distance}
\dist (A_i, \step_n^{\delta_n}) < \rho_n 
\end{equation}
by (c) and \ref{prop d}.
Since \eqref{distance} and \eqref{rho 1} hold, we can invoke Lemma \ref{lemma3} to obtain
a laminate $\nu_{A_i}$ and a rotation $R^i=R_{\theta_{A_i}}$ satisfying, in particular, $\conj{\nu}_{A_i}=A_{i}$,
\begin{gather}
|\arg R^i|=|\theta_{A_i}| < \delta_{n+1} \label{arg R} \,,  \\ 
 \spt \nu_{A_i} \subset T \cup \step_{n+1}^{\delta_{n+1}} \,,  \label{supp n+1 bis}
\end{gather}
since $\delta_{n+1}= \delta_n + \rho_n$ by \eqref{delta n}. 
By applying Proposition \ref{gradienti} to $\nu_{A_i}$ and by taking into account \eqref{supp n+1 bis}, we obtain a piecewise affine Lipschitz mapping $g_i \colon \Omega_{n,i} \to \R^2$, such that
\begin{enumerate}[resume*]
	\item \label{prop e} $g_i(x)=A_i x + b_i $ on $\partial \Omega_{n,i}$, \smallskip
	\item \label{prop f}$[g_i - f_n]_{C^{\alpha}(\overline{\Omega_{n,i}})} < 2^{-(n+1+i)} \varepsilon$, \smallskip
	\item \label{prop g}$c_K^{-1} n < | \nabla g_i (x)  | < c_K n$ a.e. in $\Omega_{n,i}$, \smallskip 
	\item \label{prop h}$\dist(\nabla g_i(x), T\cup \step_{n+1}^{\delta_{n+1}} )< \rho_{n+1} $ a.e. in $\Omega_{n,i}$.\smallskip
\end{enumerate}
Moreover 
\begin{equation} \label{stima induttiva}
 \left(  1 - c_K   \frac{\rho_n}{n} \right) \beta_n (R^i)  
 \leq \frac{|\omega_{n,i}|}{|\Omega_{n,i}|}   
\leq  \left(  1 + c_K   \frac{\rho_n}{n} \right) \beta_{n+2} (R^i) \,, 
\end{equation}
with
\[
\omega_{n,i}:=
\va{ \left\{  x \in \Omega_{n,i} \, \colon \,   
\dist( \nabla g_i (x), \step_{n+1}^{\delta_{n+1}} )< \rho_{n+1}   \right\}          } \,.
\]
Set
\begin{equation*}
f_{n+1}(x) :=
\begin{cases}
f_n (x)   &   \text{if } x \in \Omega \smallsetminus \Omega_n  \,, \\
g_i(x)    &   \text{if } x \in \Omega_{n,i} \,. 
\end{cases}
\end{equation*}
Since $\Omega_{n+1}$ is well defined, we can also introduce
\[
\tau_{n+1} (x) := 
\begin{cases}
\tau_{n}(x)   &  \text{for } x \in \Omega \smallsetminus \Omega_{n+1}\,, \\
\rho_{n+1}         &   \text{for } x \in \Omega_{n+1} \,,	
\end{cases}
\]
so that \ref{prop d} holds.
From \ref{prop e} we have $f_{n+1}(x)=Jx$ on $\partial \Omega$. From \ref{prop f} we get
$[f_{n+1} -f_n]_{C^{\alpha}(\overline{\Omega})}< 2^{-(n+1)}\varepsilon$ so that \ref{prop b} follows.
\ref{prop c} is a direct consequence of \ref{prop d}, \ref{prop h}, and the fact that $\rho_n$ is strictly decreasing.
Finally let us prove \eqref{stima omega n}. 
First notice that the sets $\omega_{n,i}$ are pairwise disjoint. By \eqref{rho 1}, in particular we have 
$\rho_{n+1} < \dist (T, \step_1)/4$, so that 
\begin{equation} \label{unione disgiunta 2}
\va{ \Omega_{n+1} \smallsetminus \bigcup_{i=1}^{\infty} \omega_{n,i}    } = 0 \,.
\end{equation} 
By \eqref{arg R} and \eqref{bound delta n} we have  
$|\arg R^i |< \delta$. Then by the properties of $\beta_n$ (see Lemma \ref{lemma5}), 
\begin{equation} \label{stima beta}
\beta_{n}(R^i) \geq \beta_{n}(R_0)    \quad \text{and} \quad \beta_{n+2}(R^i) \leq \beta_{n+2}(R_{\delta}) \,.
\end{equation}
Using \eqref{stima beta}, \eqref{unione disgiunta}, \eqref{unione disgiunta 2} in \eqref{stima omega n} yields
\[
|\Omega_n| \left(  1 - c_K   \frac{\rho_n}{n} \right) \beta_j (R_0)  
 \leq|\Omega_{n+1}|   
\leq  |\Omega_n| \left(  1 + c_K   \frac{\rho_n}{n} \right) \beta_{j+2} (R_\delta) \,,
\]
and \eqref{stima omega n} follows.

\medskip
\paragraph{\textbf{Step 2.}} 
Notice that on $\Omega \smallsetminus \Omega_n$ we have that $\nabla f_{n+1} = \nabla f_n $ almost everywhere, so  
$\Omega_{n+1} \subset \Omega_n$. Therefore $\{f_n\}$ is obtained by modification on a nested sequence of open sets, satisfying 
\[
\prod_{j=1}^{n-1} \left(  1 - c_K   \frac{\rho_j}{j} \right) \beta_j (R_0) \leq
\frac{|\Omega_n|}{|\Omega|} \leq  \prod_{j=1}^{n-1} \left(  1 + c_K   \frac{\rho_j}{j} \right) \beta_{j+2} (R_{\delta})	 \,.
\]
By \eqref{rho 1} we have $\rho_n < \min \{ 2^{-n} \, \delta, c_K^{-1}\} /4$, so that
\[
\prod_{j=1}^{\infty} \left(  1 - c_K   \frac{\rho_j}{j} \right) = c_1  \,, \quad 
 \prod_{j=1}^{\infty} \left(  1 + c_K   \frac{\rho_j}{j} \right) = c_2  \,,	
\]
with $0<c_1<c_2< \infty$, depending only on $K,S_1,S_2,\delta$ (and hence from $\delta_0$, by \eqref{delta massimo}). Moreover, from Lemma \ref{lemma5},
\[
\prod_{j=1}^n \beta_j (R_\theta)= n^{-p (R_\theta)} +O \left( \frac{1}{n} \right)\,,    \quad \text{ uniformly in } \quad (-\pi,\pi] \,.
\]
Therefore, there exists a constant $c_{K,\delta_0} >1$ depending only on $K,S_1,S_2,\delta_0$, such that
\begin{equation} \label{stima omega n 2}
\frac{1}{c_{K,\delta_0} }\, n^{-\frac{2K}{K+1}} \leq |\Omega_n| \leq 	c_{K,\delta_0} \, n^{-p_{\delta_0}} \,, 
\end{equation}
since $p(R_0) = \displaystyle \frac{2K}{K+1}$. Here $p_{\delta_0}:=p(R_\delta)$. Notice that, by \eqref{delta massimo}, $p_{\delta_0} \in \left(\frac{2K}{K+1}- \delta_0,\frac{2K}{K+1}\right]$, since $p$ is strictly decreasing in $[0,\pi/2]$.

From \eqref{stima omega n 2}, in particular we deduce $|\Omega_n| \to 0$. Therefore $f_n \to f$ almost everywhere
in $\Omega$, with $f$ piecewise affine. Furthermore $f$ satisfies (i)-(iii) by construction.

We are left to estimate the distribution function of $\nabla f$. By \ref{prop g} we have that
\[
|\nabla f(x)|> \frac{n}{c_{K,\delta_0}} \quad \text{in} \quad \Omega_n \qquad \text{and} \qquad  |\nabla f(x)|< c_{K,\delta_0} \, n \quad \text{in} \quad \Omega \smallsetminus \Omega_n \,.
\]
For a fixed $t > c_{K,\delta_0}$, let $n_1 :=[c_{K,\delta_0} t]$ and $n_2:=[c_{K,\delta_0}^{-1}t]$, where $[\cdot]$ denotes the integer part function. Therefore
\[
\Omega_{n_1+1} \subset \{ x \in \Omega \, \colon \, |\nabla f (x)|>t \}  \subset \Omega_{n_2}
\]
and \eqref{crescita} follows from \eqref{stima omega n 2}, with $p=p_{\delta_0}$. Lastly, \eqref{crescita} implies that $\nabla f_n$ is uniformly 
bounded in $L^1$, so that $f \in W^{1,1}(\Omega;\R^2)$ by dominated convergence.
\end{proof}

We remark that the constant $c_{K,\delta_0}$ in \eqref{crescita} is monotonically increasing as a function of $\delta_0$, that is $c_{K,\delta_1} \leq c_{K,\delta_2}$ if $\delta_1 \leq \delta_2$.

We now proceed with the construction of exact solutions to \eqref{differential inclusion}. We will follow a standard argument 
(see, e.g., \cite[Remark 6.3]{f}, \cite[Thoerem A.2]{npp}).



\begin{theorem} \label{thm finale}
Let $\sigma_1,\sigma_2 $ be defined by \eqref{speciali} for some $K,S_1, S_2$ as in \eqref{K bounds} and $S$ as in \eqref{def S}.   
 There exist 
coefficients $\sigma_n \in L^{\infty}(\Omega;\{ \sigma_1, \sigma_2 \})$, 
exponents $p_n \in \left[\frac{2S}{S+1},\frac{2K}{K+1} \right]$, functions $u_n \in W^{1,1} (\Omega;\R)$,
such that 
\begin{align} 
\label{pde2}
&\begin{cases}
\Div (\sigma_n (x) \nabla u_n (x)) = 0  &  \text{ in } \quad \Omega \,, \\
u_n (x) = x_1							& \text{ on }  \quad \partial \Omega \,,
\end{cases}\\
&\nabla u_n \in L^{p_n}_{\rm weak}(\Omega;\R^2), \quad p_n \to \frac{2K}{K+1}, \label{thesis1}\\
&\nabla u_n \notin L^{\frac{2K}{K+1}}(\Omega;\R^2). \label{thesis2}
\end{align}
In particular $u_n \in W^{1,q} (\Omega;\R)$ for every $q < p_n$, but $\int_{\Omega} {|\nabla u_n|}^{\frac{2K}{K+1}} \, dx= \infty$.
\end{theorem}


\begin{proof}
By Proposition \ref{prop:grad} there exist sequences $f_n \in W^{1,1} (\Omega;\R^2) \cap C^{\alpha} (\overline{\Omega};\R^2)$, $\gamma_n \searrow 0$, $p_n \in \left[\frac{2S}{S+1},\frac{2K}{K+1} \right]$, 
such that, $f_n (x)=J x$ on $\partial \Omega$,
\begin{gather} 
\dist (\nabla f_n (x) , \uno \cup \due ) < \gamma_n  \quad \text{a.e. in} \quad \Omega \,,  \label{thm:inclusion} \\ 
	\nabla f_n \in L^{p_n}_{\rm weak} (\Omega;\matrici) \,, \quad  p_n \to \frac{2K}{K+1} \,, \quad \nabla f_n \notin L^{\frac{2K}{K+1}}(\Omega;\matrici) \,.    \label{thm:reg}
\end{gather}
In euclidean coordinates, condition \eqref{thm:inclusion} implies that
\begin{equation} \label{thm:inclusion 2}
\left(
\begin{matrix}
\nabla f_n^1 (x) \\
\nabla f_n^2 (x) \\	
\end{matrix}
\right)  =
\left(
\begin{matrix}
E_n	(x) \\
\rot \sigma_n (x) E_n (x) \\
\end{matrix}
\right) +
\left(
\begin{matrix}
a_n (x)	\\
b_n (x) \\
\end{matrix}
\right) 
 \quad  \text{a.e. in} \quad \Omega
 \end{equation}
with $f_n=(f_n^1,f_n^2)$, $\sigma_n := \sigma_1\chi_{\{\nabla f \in \uno\}}  +\sigma_2\chi_{\{\nabla f \in \due\}}$, $E_n \colon \Omega \to \R^2$, $\rot=\left(
\begin{matrix}
0 & -1 \\
1    &   0 \\	
\end{matrix} \right)$ and
\begin{equation} \label{coefficienti a zero}
a_n , b_n \to 0 \qquad \text{in} \qquad L^{\infty}(\Omega;\R^2) \,.	
\end{equation} 
The boundary condition $f_n = J x$ reads $f^1_n = x_1$ and $f_n^2=-x_2$.
We set $u_n := f_n^1 + v_n$, where 
$v_n \in H^1_0 (\Omega,\R)$ is the unique solution to
\[
\Div (\sigma_n \nabla v) = - \Div (\sigma_n a_n - \rot^T b_n) \,.
\]
Notice that $v_n$ is uniformly bounded in $H^1$ by \eqref{coefficienti a zero}.
Since \eqref{thm:inclusion 2} holds, it is immediate to check that $\Div (\sigma_n \nabla u_n)= \Div (\rot^T \nabla f_n^2)=0$, so that $u_n$ is a solution of 
\eqref{pde2}. Finally, the regularity thesis \eqref{thesis1}, \eqref{thesis2}, follows from the definition of $u_n$ and the fact that $v_n \in H_0^1(\Omega;\R)$ and $f_n^1$ satisfies \eqref{thm:reg} with $1<p_n<2$.
\end{proof}

\nocite{*}
\bibliography{bibliografia}

\bibliographystyle{plain}

\end{document}